\documentclass[a4paper,fleqn]{cas-sc}

\usepackage[numbers]{natbib}

\usepackage{amsmath,amssymb,amsthm,mathtools}
\usepackage{booktabs}
\usepackage{enumitem}
\usepackage{hyperref}
\usepackage{xcolor}
\usepackage{microtype}
\usepackage{tikz}
\usepackage{subcaption}
\usepackage{caption}
\usetikzlibrary{positioning, arrows.meta, decorations.pathreplacing}

\usepackage{placeins}   

\ExplSyntaxOn
\cs_set:Npn \__first_footerline: {}
\ExplSyntaxOff

\ExplSyntaxOn
\bool_gset_true:N \g_stm_nologo_bool
\ExplSyntaxOff

\newtheorem{theorem}{Theorem}[section]
\newtheorem{proposition}[theorem]{Proposition}
\newtheorem{lemma}[theorem]{Lemma}
\newtheorem{corollary}[theorem]{Corollary}
\theoremstyle{definition}

\newtheorem{remark}{Remark}[section]
\newtheorem{example}[theorem]{Example}

\newcommand{\D}{\mathcal D}
\newcommand{\F}{\mathcal F}
\newcommand{\Des}{\operatorname{Des}}
\newcommand{\LE}{\operatorname{LE}}
\newcommand{\st}{\operatorname{st}}

\hypersetup{
	colorlinks=true,
	linkcolor=red,
	citecolor=blue,
	urlcolor=green
}

\def\tsc#1{\csdef{#1}{\textsc{\lowercase{#1}}\xspace}}
\tsc{WGM}
\tsc{QE}
\tsc{EP}
\tsc{PMS}
\tsc{BEC}
\tsc{DE}

\begin{document}
	\let\WriteBookmarks\relax
	\def\floatpagepagefraction{1}
	\def\textpagefraction{.001}
	
	\shorttitle{Pattern avoidance in generalized alternating permutations}
	
	\shortauthors{Luo et al.}
	
	\title[mode = title]{Classical and vincular patterns of length three in generalized alternating permutations}
	
	\tnotemark[1]
	
	\tnotetext[1]{This work was supported by the National Science Foundation of China (grant 12201641).}
	
	\author[1]{Zhenhua Luo}[
	auid=000,
	bioid=1,
	orcid=0009-0007-5106-7490
	]
	\fnmark[1]   
	\ead{luozhenhua2004@126.com}
	\credit{Methodology, Investigation}
	
	\author[2]{Junting Wang}
	\fnmark[2]
	\ead{2022011749@student.cup.edu.cn}
	\credit{Investigation, Validation}
	
	\author[3]{Ziyi Yang}[
	auid=003,
	bioid=3,
	orcid=0009-0007-2686-7509
	]
	
	\fnmark[3]
	\ead{yangziyi.2000@outlook.com}
	\credit{Writing – review \& editing}
	
	\author[3]{Feng Zhao}[
	auid=004,
	bioid=4,	
	]
	\cormark[1]   
	\ead{zhaofeng@hebtu.edu.cn}
	\credit{Conceptualization, Supervision, Project administration}
	
	\author[2]{Tongyuan Zhao}[
	auid=005,
	bioid=5,
	orcid=0000-0003-3473-1445
	]
	\fnmark[5]
	\ead{zhaotongyuan@cup.edu.cn}
	\credit{Writing – Original Draft, Project Administration, Supervision}
	
	\cortext[1]{Corresponding author}
	
	
	\affiliation[1]{organization={School of Mathematics and Information Science},
		addressline={Guangzhou University},
		city={Guangzhou},
		country={China}}
	\affiliation[2]{organization={College of Science},
		addressline={China University of Petroleum},
		city={Beijing},
		postcode={102249},
		country={China}}
	\affiliation[3]{organization={School of Mathematical Sciences},
		addressline={Hebei Normal University},
		city={Shijiazhuang},
		postcode={050024},
		country={China}}
	
	\begin{abstract}
		Let $k\ge2$ and let $\mathcal{D}_{N,k}$ be the set of permutations of $[N]={\{1,2,\ldots,N\}}$ whose descent set is exactly $\{k,2k,\ldots,k\lfloor (N-1)/k\rfloor\}$. We enumerate the elements of $\mathcal{D}_{N,k}$ avoiding each classical and each vincular pattern of length three.
		
		For classical patterns, we give recursive bijections from the $132$- and $231$-avoiding classes to ordered forests of $s$ complete $k$-ary trees, obtaining the Raney number. The $213$- and $312$-avoiding classes are obtained from these forest bijections by completing the last block and applying reverse-complement symmetry. The remaining classical case $321$ is expressed by RSK.
		
		For vincular patterns, we enumerate the six fully consecutive patterns and the twelve patterns with exactly one adjacency. The closed results forms are accompanied by bijective models: the Catalan, Raney, Fuss-Catalan, and RSK cases are natural $k$-ary or fixed-descent extensions of classical bijections, while the product and poset cases arise from block-insertion and record/tree-poset encodings forced by the adjacency conditions.
	\end{abstract}
	
	
	
	\maketitle
	
	
	\section{Introduction}
	
	A permutation \(\pi=\pi_1\cdots\pi_N\) contains a pattern \(\sigma\in S_r\) if a subsequence of \(\pi\) has the same relative order as \(\sigma\); otherwise \(\pi\) avoids \(\sigma\).  Pattern avoidance first arose naturally in Knuth's study of stack sorting, where the avoidance of a single pattern of length three leads to Catalan numbers \cite{KnuthTAOCP}.  Simion and Schmidt gave the classical systematic treatment of restrictions by patterns of length three, including bijections, parity refinements, and simultaneous avoidance classes \cite{SimionSchmidt}.  The subject has since developed into a central part of enumerative combinatorics; standard references include the books of B\'ona and Kitaev \cite{BonaBook,KitaevBook}.
	
	There are two refinements of ordinary pattern avoidance that are especially relevant here.  Babson and Steingr\'imsson introduced generalized permutation patterns, now often called vincular patterns, in which selected adjacent letters of the pattern must be realized by adjacent entries in the ambient permutation \cite{BabsonSteingrimsson}.  Claesson developed their enumerative theory and terminology \cite{Claesson}.  The case in which all adjacent positions are prescribed is the theory of consecutive patterns; systematic enumeration of consecutive patterns was initiated by Elizalde and Noy and later surveyed by Elizalde \cite{ElizaldeNoy,ElizaldeSurvey}.  Vincular patterns of length three are therefore the first nontrivial meeting point between classical and consecutive avoidance.
	
	Alternating permutations form another natural restricted setting for pattern avoidance.  Stanley's survey records many enumerative properties of alternating permutations and their relation to descent sets, Euler numbers, and tableaux \cite{StanleySurvey}.  Mansour studied restricted \(132\)-alternating permutations via Chebyshev polynomials \cite{Mansour}.  In a series of works, Lewis proved that avoiding any single pattern of length three in alternating permutations gives Catalan enumerations, developed further connections with Young tableaux, RSK-type correspondences, generating trees, and lattice paths for longer patterns \cite{LewisNote,LewisJCTA,LewisEJC,LewisDMTCS}.
	
	Several of Lewis's conjectures and equivalences were later studied by B\'ona, Chen and Zhou and Xu and Yan\cite{Bona,ChenChenZhou,XuYan}.  Around the same time, to address deeper structural recurrences and equivalences, Gowravaram and Jagadeesan introduced the framework of ascent–descent (AD) Young diagrams and established the shape-Wilf equivalences \(12q\sim21q\) and \(213q\sim321q\) for alternating permutations, thereby providing a unified explanation for several Catalan-type enumerations \cite{GowravaramJagadeesan,Jagadeesan}.
	
	This paper considers a periodic-descent-set analogue of alternating permutations. This perspective is closely related to the reading words of ``thickened staircase tableaux'' appearing in the early work of Lewis \cite{LewisDMTCS}, and was later extended by Gowravaram and Jagadeesan in their studies of pattern avoidance in Young diagrams, ascent-descent Young diagrams, and permutations with restricted ascents and descents \cite{GowravaramJagadeesan,Jagadeesan}. Together, these works place alternating pattern avoidance in a broader framework in which the descent set itself, rather than alternation alone, becomes a fundamental part of the structure. In fact, from this broader perspective, the AD-Young diagram framework already suggests that such generalized descent classes may exhibit rich Wilf-equivalence phenomena.
	
	More concretely, Lewis explicitly raised the question of whether the generalized descent classes \(\mathrm{Des}_{n,k}\)---permutations whose descent set is exactly \(\{k, 2k, \ldots\}\)---provide a natural setting for pattern avoidance, and whether nontrivial Wilf equivalences exist within these classes \cite{LewisEJC}. These \(\mathrm{Des}_{n,k}\) classes are precisely the sets \(\mathcal{D}_{N,k}\) studied in this paper.
	
	In this paper, we initiate a systematic study of pattern avoidance in the generalized alternating class \(\mathcal{D}_{N,k}\). For \(k\ge2\) and \(N\ge1\), define
	\[
	\D_{N,k}=\{\pi\in S_N:\Des(\pi)=\{k,2k,\ldots,k\lfloor (N-1)/k\rfloor\}\}.
	\]
	Equivalently, entries increase inside each block of length \(k\), while there is a descent at every boundary between consecutive blocks.  The case \(k=2\) is the usual up-down alternating class.  We write throughout
	\[
	N=mk+s,\qquad 1\le s\le k,
	\]
	so the first \(m\) blocks have length \(k\) and the last block has length \(s\).  This convention also treats complete blocks uniformly: if \(N=nk\), then \(m=n-1\) and \(s=k\).
	
	Our first result is a complete enumeration of \(\D_{N,k}\) avoiding a single classical pattern of length three.  The patterns \(132\) and \(231\) are counted by the Raney number
	\[
	\frac{s}{mk+s}\binom{mk+s}{m},
	\]
	which is the number of ordered forests of \(s\) complete \(k\)-ary trees with \(m\) internal vertices.  The patterns \(213\) and \(312\) are counted by the Fuss-Catalan number
	\[
	\frac{1}{(k-1)(m+1)+1}\binom{k(m+1)}{m+1}.
	\]
	The increasing pattern \(123\) is degenerate when \(k\ge3\), while the decreasing pattern \(321\) is naturally described by the Robinson-Schensted correspondence in terms of two-row standard Young tableaux and ballot words.
	
	Our second result is a classification of all vincular patterns of length three in \(\D_{N,k}\).  We treat separately the six fully consecutive patterns and the twelve patterns with exactly one adjacency.  The answers fall into four types: degenerate cases, Raney or Fuss-Catalan numbers, explicit factorial-product formulas, and linear-extension numbers of concrete posets obtained by adding natural inequalities to the block poset.  The linear-extension cases are included as exact formulas rather than recurrences; they mark precisely the cases where local vincular avoidance imposes a global partial order but does not simplify to a product.
	
	The paper is organized so that the bijective structure comes before the algebra.  Section~2 gives notation, the tree generating functions, and the reverse-complement symmetry used in complete blocks.  Section~3 proves the classical length-three enumerations by explicit forest, tree, and RSK models.  Section~4 treats fully consecutive patterns by local boundary constraints.  Section~5 treats all one-adjacent vincular patterns by tightening, insertion, and poset arguments.  
	
	\section{Definitions and preliminary objects}
	
	For a permutation $\pi=\pi_1\cdots\pi_N$, let
	\[
	\Des(\pi)=\{i\in[N-1]:\pi_i>\pi_{i+1}\}.
	\]
	For $N=mk+s$ with $1\le s\le k$, a permutation in $\D_{N,k}$ is written in blocks
	\[
	B_1B_2\cdots B_mB_{m+1},
	\]
	where the first $m$ blocks have length $k$ and the last block has length $s$.  We write the entries in the $i$th complete block as
	\[
	x_{i,1}<x_{i,2}<\cdots <x_{i,k}\qquad (1\le i\le m),
	\]
	and the entries in the last block as
	\[
	x_{m+1,1}<\cdots <x_{m+1,s}.
	\]
	The boundary descent relations are
	\[
	x_{i,k}>x_{i+1,1}\qquad (1\le i\le m).
	\]
	Equivalently, if values are read as labels of the positions, the defining poset for $\D_{N,k}$, denoted $B_{m,s,k}$, has the relations
	\[
	x_{i,a}<x_{i,a+1}\quad\text{inside each block},\qquad
	x_{i+1,1}<x_{i,k}\quad(1\le i\le m).
	\]
	
	\begin{lemma}\label{lem:blockposet}
		The linear extensions of \(B_{m,s,k}\) are precisely the permutations in \(\D_{N,k}\).
	\end{lemma}
	
	\begin{proof}
		A linear extension assigns the values \(1,2,\ldots,N\) to the positions in a way that is increasing along all displayed relations.  Hence entries increase inside each block and satisfy \(x_{i,k}>x_{i+1,1}\) at every boundary.  These are exactly the required ascents and descents.  Conversely, every permutation in \(\D_{N,k}\) satisfies these relations, and so gives a linear extension of \(B_{m,s,k}\).
	\end{proof}
	
	A vincular pattern of length three is a pattern in $S_3$ together with a set of adjacencies.  We use the notation $\underline{12}3$ for a pattern in which the entries realizing the first two letters must be adjacent, $1\underline{23}$ for a pattern in which the entries realizing the last two letters must be adjacent, and $\underline{123}$ for a fully consecutive pattern.  We write $\D_{N,k}(p)$ for the set of elements of $\D_{N,k}$ avoiding the pattern $p$.
	
	Let $T(z)$ be the generating function for complete ordered $k$-ary trees by internal vertices:
	\[
	T(z)=1+zT(z)^k.
	\]
	Then $T(z)^s$ is the generating function for ordered forests of $s$ complete $k$-ary trees.  Lagrange inversion gives the Raney number \cite{Raney,StanleyEC2}
	\begin{equation}\label{eq:Raney}
		[z^m]T(z)^s=\frac{s}{mk+s}\binom{mk+s}{m}.
	\end{equation}
	We shall also use the Fuss-Catalan number
	\begin{equation}\label{eq:Fuss}
		C^{(k)}_{m+1}=\frac{1}{(k-1)(m+1)+1}\binom{k(m+1)}{m+1}.
	\end{equation}
	
	We shall also use the following symmetry in the complete-block case.  For a permutation
	\(\pi=\pi_1\cdots\pi_N\), let
	\[
	\operatorname{rc}(\pi)_i=N+1-\pi_{N+1-i}
	\]
	be its reverse-complement.
	
	\begin{lemma}\label{lem:rc}
		If $N=nk$, then $\operatorname{rc}$ is an involution of $\D_{N,k}$.  It sends avoidance of a classical or vincular pattern $p$ to avoidance of the reverse-complement pattern $\operatorname{rc}(p)$, with the adjacencies reversed.
	\end{lemma}
	
	\begin{proof}
		The descent set of $\operatorname{rc}(\pi)$ is the reflection of the descent set of $\pi$.  Since
		$\{k,2k,\ldots,(n-1)k\}$ is invariant under reflection in $[N-1]$, the map preserves $\D_{N,k}$.  The assertion about patterns follows directly from reversing the three chosen positions and complementing their relative order; adjacency between consecutive chosen positions is preserved, but its side is reversed.
	\end{proof}
	
	\section{Classical patterns of length three}
	
	\subsection{The patterns $132$ and $231$}
	
	Let $\F^{(k)}_{m,s}$ denote the set of ordered forests of $s$ complete ordered $k$-ary trees with altogether $m$ internal vertices.  Thus
	\[
	|\F^{(k)}_{m,s}|=[z^m]T(z)^s
	=\frac{s}{mk+s}\binom{mk+s}{m}.
	\]
	
	\begin{theorem}\label{thm:classical132231}
		Let $k\ge2$, $m\ge0$, $1\le s\le k$, and $N=mk+s$.  There are recursive bijections
		\[
		\D_{N,k}(132)\longleftrightarrow \F^{(k)}_{m,s},
		\qquad
		\D_{N,k}(231)\longleftrightarrow \F^{(k)}_{m,s}.
		\]
		Consequently,
		\[
		|\D_{N,k}(132)|=|\D_{N,k}(231)|=\frac{s}{N}\binom{N}{m}.
		\]
	\end{theorem}
	
	\begin{proof}
		We give the bijection for $132$ first.  If $m=0$, then the only permutation is $12\cdots s$, and it is sent to the forest of $s$ leaves.  Now let $m>0$ or $s>1$, and let $N=mk+s$ be the maximum entry.  Since entries increase inside each block, $N$ is the last entry of a block.
		
		If $\pi\in\D_{N,k}(132)$ and $N$ is at position $p$, then every entry to the left of $N$ is larger than every entry to the right of $N$.  Otherwise there would be entries $a$ to the left and $c$ to the right with $a<c<N$, giving a $132$ occurrence $a,N,c$.
		
		If $N$ is the last entry of the final block and $s\ge2$, delete $N$.  This gives an element $\pi'\in\D_{mk+s-1,k}(132)$.  The corresponding forest is obtained by adjoining one leaf in front of the forest associated with $\pi'$:
		\[
		\Phi^{132}_{m,s}(\pi)=(\bullet,\Phi^{132}_{m,s-1}(\pi')).
		\]
		
		Otherwise $N$ is the last entry of the $i$th complete block.  Write
		\[
		\pi=L\,N\,R,
		\]
		where $L$ consists of the preceding $(i-1)$ complete blocks and the first $k-1$ entries of the $i$th block, while $R$ is the remaining suffix.  The condition above forces all values in $L$ to be larger than all values in $R$.  Hence standardization gives
		\[
		\st(L)\in\D_{(i-1)k+k-1,k}(132),
		\qquad
		\st(R)\in\D_{(m-i)k+s,k}(132).
		\]
		Write
		\[
		\Phi^{132}_{i-1,k-1}(\st(L))=(T_1,\ldots,T_{k-1})
		\]
		and
		\[
		\Phi^{132}_{m-i,s}(\st(R))=(T_k,U_2,\ldots,U_s).
		\]
		Then $\Phi^{132}_{m,s}(\pi)$ is the forest whose first tree has a new root with ordered children $T_1,\ldots,T_k$, followed by $U_2,\ldots,U_s$.
		
		This construction is invertible.  If the first tree of a forest is a leaf, remove it and append the new maximum at the end of the final block.  If the first tree is rooted with children $T_1,\ldots,T_k$, reconstruct a left factor from $(T_1,\ldots,T_{k-1})$ and a right factor from $(T_k,U_2,\ldots,U_s)$; inflate the right factor by the smaller value interval, inflate the left factor by the next larger interval, and place the maximum between them.  The inequalities just described prevent cross $132$ occurrences, and the recursive factors avoid $132$, so the inverse is well defined.
		
		For $231$, the same recursive decomposition is used, but a $231$-avoiding permutation with maximum $N$ has every entry to the left of $N$ smaller than every entry to the right of $N$.  Thus the inverse is the same except that the left recursive factor is inflated by the smaller interval and the right recursive factor by the larger interval.  The same argument proves that the two constructions are inverse bijections.  Counting the forests by \eqref{eq:Raney} completes the proof.
	\end{proof}
	
	\begin{remark}\label{rem:raneypathbijection}
		Composing either forest bijection with the standard Lukasiewicz traversal gives a path model.  A forest in $\F^{(k)}_{m,s}$ is read from left to right in depth-first order; each internal vertex contributes an up-step of size $k-1$, and each leaf contributes a down-step of size $1$.  The path starts at height $s$, has $m$ up-steps and $(k-1)m+s$ down-steps, stays nonnegative, and ends at height $0$.  This is the usual Raney-path model, now attached directly to the avoiding permutations.
	\end{remark}
	
\begin{example}
	To illustrate the recursive bijection of Theorem~\ref{thm:classical132231}, we give two concrete examples with $k=3$, $m=1$, $s=2$ (so $N=5$).  The class $\D_{5,3}(132)$ has exactly $R^{(3)}_{1,2}=\frac{2}{5}\binom{5}{1}=2$ permutations: $23415$ and $34512$.  
	Figure~\ref{fig:examples} shows the mapping for these two permutations.  Subfigure~\ref{fig:23415} corresponds to $\pi=23415$, and Subfigure~\ref{fig:34512} corresponds to $\pi=34512$.  In each case, the original permutation is shown at the top with its block structure; a downward arrow labeled ``recursive decomposition'' indicates the application of the bijection; the resulting ordered forest of two complete $3$-ary trees with one internal vertex is displayed at the bottom.  The labels on the leaves indicate the original values from which they derive.
\end{example}
\FloatBarrier   
	\begin{figure}[htbp]
	\centering
	\begin{subfigure}{0.48\textwidth}
		\centering
		\begin{tikzpicture}[
			scale=0.65,
			>=Stealth,
			block/.style={rectangle, draw, thick, rounded corners=2pt, minimum height=0.5cm, minimum width=0.6cm, align=center, font=\footnotesize},
			treenode/.style={circle, draw, thick, minimum size=0.4cm, inner sep=0pt, font=\footnotesize},
			]
			\node[block, fill=blue!20] at (0,0) {2};
			\node[block, fill=blue!20] at (0.9,0) {3};
			\node[block, fill=blue!20] at (1.8,0) {4};
			\node[block, fill=green!20] at (3.0,0) {1};
			\node[block, fill=green!20] at (3.9,0) {5};
			
			\draw[->, thick] (2.0,-0.6) -- (2.0,-1.8) node[midway, right, font=\tiny] {recursive decomposition};
			
			\begin{scope}[yshift=-2.8cm, scale=1.2, every node/.append style={font=\footnotesize}]
				\begin{scope}[level distance=0.8cm, sibling distance=0.8cm]
					\node[treenode, label=above:{Tree 1}] (root1) {}
					child { node[treenode, label=below:{2}] {} }
					child { node[treenode, label=below:{3}] {} }
					child { node[treenode, label=below:{1}] {} };
				\end{scope}
				\node[treenode, label=below:{5}, right=2.2cm of root1] (leaf) {};
				\node[above=0.2cm of leaf, font=\tiny] {Tree 2};
			\end{scope}
		\end{tikzpicture}
		\caption{The permutation $23415$ and its image under the recursive bijection.}
		\label{fig:23415}
	\end{subfigure}
	\hfill
	\begin{subfigure}{0.48\textwidth}
		\centering
		\begin{tikzpicture}[
			scale=0.65,
			>=Stealth,
			block/.style={rectangle, draw, thick, rounded corners=2pt, minimum height=0.5cm, minimum width=0.6cm, align=center, font=\footnotesize},
			treenode/.style={circle, draw, thick, minimum size=0.4cm, inner sep=0pt, font=\footnotesize},
			]
			\node[block, fill=blue!20] at (0,0) {3};
			\node[block, fill=blue!20] at (0.9,0) {4};
			\node[block, fill=blue!20] at (1.8,0) {5};
			\node[block, fill=green!20] at (3.0,0) {1};
			\node[block, fill=green!20] at (3.9,0) {2};
			
			\draw[->, thick] (2.0,-0.6) -- (2.0,-1.8) node[midway, right, font=\tiny] {recursive decomposition};
			
			\begin{scope}[yshift=-2.8cm, scale=1.2, every node/.append style={font=\footnotesize}]
				\begin{scope}[level distance=0.8cm, sibling distance=0.8cm]
					\node[treenode, label=above:{Tree 1}] (root1) {}
					child { node[treenode, label=below:{3}] {} }
					child { node[treenode, label=below:{4}] {} }
					child { node[treenode, label=below:{1}] {} };
				\end{scope}
				\node[treenode, label=below:{2}, right=2.2cm of root1] (leaf) {};
				\node[above=0.2cm of leaf, font=\tiny] {Tree 2};
			\end{scope}
		\end{tikzpicture}
		\caption{The permutation $34512$ and its image under the recursive bijection.}
		\label{fig:34512}
	\end{subfigure}
	\caption{Illustration of the recursive bijection in Theorem~\ref{thm:classical132231} for the two permutations in $\mathcal{D}_{5,3}(132)$. Each permutation maps to a forest of two complete $3$-ary trees with one internal vertex.}
	\label{fig:examples}
\end{figure}
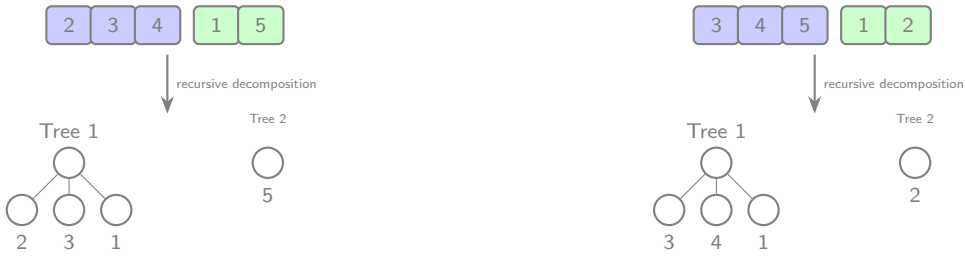
\FloatBarrier

	\subsection{The patterns $213$ and $312$}
	
	The following two elementary completion lemmas are the point at which the incomplete last block must be treated carefully.
	
	\begin{lemma}\label{lem:312completion}
		Let $\pi\in\D_{mk+s,k}(312)$.  Then the entries $x_{m+1,2},\ldots,x_{m+1,s}$, when present, are the largest $s-1$ entries of $\pi$.  Consequently, appending one new largest entry to the last block gives a bijection
		\[
		\D_{mk+s,k}(312)\longleftrightarrow \D_{mk+s+1,k}(312)
		\]
		for $1\le s<k$.
	\end{lemma}
	
	\begin{proof}
		If $1\le a<s$ and some earlier entry $y$ were larger than $x_{m+1,a+1}$, then
		\[
		y,\ x_{m+1,a},\ x_{m+1,a+1}
		\]
		would form a $312$ pattern.  Hence no entry before the last block is larger than $x_{m+1,a+1}$.  Since the last block is increasing, the entries after $x_{m+1,1}$ in the last block must therefore be precisely the largest $s-1$ entries.
		
		Appending a new largest entry cannot create a new $312$ pattern: the new entry is in the final position, so it cannot be the first or second element of such a pattern, and it cannot be the third element because no earlier entry is larger than it.  Deleting the final largest entry is the inverse operation and plainly preserves avoidance.
	\end{proof}
	
	\begin{lemma}\label{lem:213completion}
		Let $\pi\in\D_{mk+s,k}(213)$.  Then the final block $B_{m+1}$ consists of consecutive values.  Consequently, for $1\le s<k$, inserting after $x_{m+1,s}$ a new value immediately above $x_{m+1,s}$ in value order gives a bijection
		\[
		\D_{mk+s,k}(213)\longleftrightarrow \D_{mk+s+1,k}(213).
		\]
		Equivalently, if the old last entry is $b$, replace the value set by $[N+1]$ by inserting the new last entry $b+1$ and increasing every old value greater than $b$ by one.
	\end{lemma}
	
	\begin{proof}
		If the final block had a gap, then for some $a<s$ there would be an earlier entry $y$ satisfying
		\[
		x_{m+1,a}<y<x_{m+1,a+1}.
		\]
		The three entries $y,x_{m+1,a},x_{m+1,a+1}$ would then form a $213$ pattern.  Thus the values in the final block are consecutive.
		
		It remains to check the insertion.  No old occurrence is created because the relative order of all old entries is unchanged.  Suppose a new $213$ occurrence is created.  Since the inserted value $b+1$ is in the final position, it must be the third, largest entry of the occurrence.  Thus there are old entries $a_1,a_2$ in positions $i<j$ with
		\[
		a_2<a_1<b+1.
		\]
		The value $a_1$ cannot be the old final entry $b$, because $b$ occurs after $a_2$.  Hence $a_1<b$, and then $a_1,a_2,b$ already formed a $213$ pattern before insertion, a contradiction.  Deletion of the last entry is the inverse operation and preserves avoidance.
	\end{proof}
	
	\begin{theorem}
		\label{thm:classical213312}
		Let \(k\ge2\), \(m\ge0\), \(1\le s\le k\), and \(N=mk+s\).  The classes \(\D_{N,k}(213)\) and \(\D_{N,k}(312)\) are in explicit bijection with complete ordered \(k\)-ary trees with \(m+1\) internal vertices.  Consequently,
		\[
		|\D_{N,k}(213)|=|\D_{N,k}(312)|=C^{(k)}_{m+1}.
		\]
	\end{theorem}
	
	\begin{proof}
		For $312$, apply Lemma~\ref{lem:312completion} repeatedly until the last block has length $k$.  This gives a bijection
		\[
		\D_{mk+s,k}(312)\longleftrightarrow \D_{(m+1)k,k}(312).
		\]
		Apply reverse-complement, which sends the complete-block $312$-class to the complete-block $231$-class.  Then apply the forest bijection of Theorem~\ref{thm:classical132231} with parameters $(m,k)$.  The result is a forest of \(k\) complete ordered \(k\)-ary trees with altogether \(m\) internal vertices.  Adding a common root whose children are these \(k\) trees gives a complete ordered \(k\)-ary tree with \(m+1\) internal vertices.
		
		For $213$, use Lemma~\ref{lem:213completion} to complete the final block, apply reverse-complement to pass from complete-block $213$-avoidance to complete-block $132$-avoidance, and then use the forest bijection of Theorem~\ref{thm:classical132231}.  Again, adjoining a common root to the resulting forest of \(k\) trees gives a complete ordered \(k\)-ary tree with \(m+1\) internal vertices.
		
		Each step in both constructions is reversible: delete the common root, apply the appropriate forest inverse, apply reverse-complement, and finally undo the completion of the last block.  Thus we have explicit bijections.  The number of complete ordered \(k\)-ary trees with \(m+1\) internal vertices is the Fuss-Catalan number in \eqref{eq:Fuss}.
	\end{proof}
	
	\begin{remark}\label{rem:fussdyck}
		The preorder Lukasiewicz traversal of the resulting complete \(k\)-ary tree is the usual \(k\)-Dyck path with \(m+1\) up-steps.  Thus the maps above are not merely enumerative reductions: they identify each $213$- or $312$-avoider with its Fuss-Catalan tree and path.
	\end{remark}
	
	\subsection{The patterns $123$ and $321$}
	
	\begin{proposition}\label{prop:classical123}
		If $k\ge3$ and $N\ge3$, then $\D_{N,k}(123)=\varnothing$.  If $k=2$, then
		\[
		|\D_{2m+s,2}(123)|=C_{m+1},\qquad s=1,2.
		\]
	\end{proposition}
	
	\begin{proof}
		For $k\ge3$ and $N\ge3$, the first three entries lie in the first block and are increasing, hence contain a classical $123$ pattern.  For $k=2$, this is the classical alternating case, and the stated Catalan enumeration is due to Lewis \cite{LewisNote}.
	\end{proof}
	
	For $321$, the natural answer is expressed by RSK.  Let $K_{N,k}(a)$ be the number of ballot words of length $N$ with $a$ letters $D$ such that peaks $UD$ occur exactly at the boundary positions $k,2k,\ldots,km$.
	
	\begin{theorem}\label{thm:classical321}
		Let $N=mk+s$.  Then
		\[
		|\D_{N,k}(321)|=
		\sum_{a=0}^{\lfloor N/2\rfloor}K_{N,k}(a)\frac{N-2a+1}{N-a+1}\binom{N}{a}.
		\]
	\end{theorem}
	
	\begin{proof}
		Under the Robinson-Schensted correspondence \cite{Schensted,KnuthRSK}, $321$-avoiding permutations correspond to pairs of standard Young tableaux of the same shape with at most two rows.  Descents of the permutation are descents of the recording tableau.  A two-row recording tableau is encoded by a ballot word, with a peak $UD$ at $i$ precisely when $i$ is a descent.  Thus the recording tableaux are counted by $K_{N,k}(a)$, and for each such recording tableau the insertion tableau can be chosen in
		\[
		f^{(N-a,a)}=\frac{N-2a+1}{N-a+1}\binom{N}{a}
		\]
		ways.  Summing over $a$ proves the formula.
	\end{proof}
	
	\begin{corollary}\label{cor:classicalgf}
		Let \(T(z)=1+zT(z)^k\).  Then
		\[
		\sum_{m\ge0}|\D_{mk+s,k}(132)|z^m
		=\sum_{m\ge0}|\D_{mk+s,k}(231)|z^m=T(z)^s,
		\]
		and
		\[
		\sum_{m\ge0}|\D_{mk+s,k}(213)|z^m
		=\sum_{m\ge0}|\D_{mk+s,k}(312)|z^m=\frac{T(z)-1}{z}.
		\]
		In particular, in the complete-block case \(N=(m+1)k\), the four nonmonotone classical classes
		\(132,231,213,312\) all have cardinality \(C^{(k)}_{m+1}\).
	\end{corollary}
	
	\begin{proof}
		The first identity is just Theorem~\ref{thm:classical132231}.  The second follows from
		Theorem~\ref{thm:classical213312} and the identity
		\([z^{m+1}]T(z)=C^{(k)}_{m+1}\), equivalently \(T(z)-1=zT(z)^k\).  The complete-block assertion is the specialization $s=k$, for which the Raney number in \eqref{eq:Raney} becomes the Fuss-Catalan number in \eqref{eq:Fuss}.
	\end{proof}
	
	\section{Fully consecutive patterns}
	
	A fully consecutive pattern is denoted by underlining all three letters.  Thus $\underline{312}$ means that three adjacent entries have relative order $312$.
	
	We shall use twice the following elementary count.  Start with a chain
	\(c_1<c_2<\cdots<c_L\) and add incomparable elements \(u_1,\ldots,u_m\).  If the only additional restrictions are \(u_i<c_{a_i}\), where
	\(1\le a_1\le\cdots\le a_m\), then the number of linear extensions is
	\[
	\prod_{i=1}^m(a_i+i-1).
	\]
	Indeed, insert the elements \(u_1,u_2,\ldots,u_m\) successively into the chain; when inserting \(u_i\), it must be placed before \(c_{a_i}\), and the previous \(i-1\) inserted elements are already among the available earlier positions.  Dually, if the only restrictions are \(c_{b_i}<u_i\), the same argument applied from the top of the chain gives the analogous product of the numbers of available positions above \(c_{b_i}\).
	
	\begin{theorem}\label{thm:fullyconsecutive}
		Let $N=mk+s$, $1\le s\le k$.  The fully consecutive patterns of length three in $\D_{N,k}$ are enumerated as follows.
		\[
		\begin{array}{c|c}
			\toprule
			\text{pattern} & \text{number of avoiding permutations}\\
			\midrule
			\underline{123} & 0\text{ if }k\ge3\text{ and }N\ge3;\ \text{all of }\D_{N,k}\text{ if }N<3\text{ or }k=2\\[2mm]
			\underline{321} & |\D_{N,k}|\\[1mm]
			\underline{312} & k^m m!\\[1mm]
			\underline{231} & \displaystyle\prod_{j=0}^{m-1}(s+jk)\\[3mm]
			\underline{132} & \LE(P^{132}_{m,s,k})\\[1mm]
			\underline{213} & \LE(P^{213}_{m,s,k})\\
			\bottomrule
		\end{array}
		\]
		Here $\LE(P)$ denotes the number of linear extensions of the poset $P$.  The posets $P^{132}_{m,s,k}$ and $P^{213}_{m,s,k}$ are obtained from the block poset \(B_{m,s,k}\) by adding, respectively, the relations
		\[
		x_{i+1,1}<x_{i,k-1}\qquad (1\le i\le m),
		\]
		whenever $x_{i,k-1}$ exists, and
		\[
		x_{i+1,2}<x_{i,k}\qquad (1\le i\le m),
		\]
		whenever $x_{i+1,2}$ exists.
	\end{theorem}
	
	\begin{proof}
		The cases $\underline{123}$ and $\underline{321}$ are immediate from the block structure: if $N\ge3$ and the first block has length at least three, then its first three entries form a consecutive $123$; when $k=2$ every consecutive triple contains exactly one ascent and one descent.  Also, the descent set has no two consecutive descents, so a consecutive $321$ is impossible.
		
		For $\underline{312}$, a consecutive occurrence can only straddle a boundary and has the form
		\[
		x_{i,k},\ x_{i+1,1},\ x_{i+1,2}.
		\]
		It is a $312$ precisely when $x_{i,k}>x_{i+1,2}$.  Hence avoidance is equivalent to the single chain
		\[
		x_{1,1}<\cdots<x_{1,k}<x_{2,2}<\cdots<x_{2,k}<\cdots <x_{m+1,2}<\cdots<x_{m+1,s},
		\]
		together with the $m$ additional elements $x_{2,1},\ldots,x_{m+1,1}$, where $x_{i+1,1}$ is required to be below the chain element $x_{i,k}$.  In this chain, the element $x_{i,k}$ is in position $i(k-1)+1$.  The chain-insertion count above therefore gives
		\[
		\prod_{i=1}^m\bigl(i(k-1)+1+i-1\bigr)=\prod_{i=1}^m ik=k^m m!.
		\]
		
		For $\underline{231}$, a consecutive occurrence at the boundary between blocks $i$ and $i+1$ has the form
		\[
		x_{i,k-1},\ x_{i,k},\ x_{i+1,1}
		\]
		and is a $231$ precisely when $x_{i,k-1}>x_{i+1,1}$.  Avoidance is therefore equivalent to the chain obtained by deleting the last entry of each complete block, together with each deleted entry $x_{i,k}$ placed above the chain element $x_{i+1,1}$.  Inserting these deleted entries from right to left gives respectively
		\[
		s, s+k,\ldots, s+(m-1)k
		\]
		available positions, and hence the product $\prod_{j=0}^{m-1}(s+jk)$.
		
		Finally, a consecutive triple not contained in a single block can only straddle a boundary.  The pattern $\underline{132}$ occurs at the boundary between blocks $i$ and $i+1$ exactly when
		\[
		x_{i,k-1}<x_{i+1,1}<x_{i,k}.
		\]
		Since $x_{i+1,1}<x_{i,k}$ is already the boundary descent, avoiding $\underline{132}$ is equivalent to imposing $x_{i+1,1}<x_{i,k-1}$.  Similarly, $\underline{213}$ occurs at a boundary exactly when $x_{i,k}<x_{i+1,2}$, so avoiding it is equivalent to imposing $x_{i+1,2}<x_{i,k}$.  Thus the avoiding permutations are precisely the linear extensions of the two displayed posets.
	\end{proof}
	
	\section{Vincular patterns with exactly one adjacency}
	
	We now classify the twelve vincular patterns of length three with exactly one adjacency.  For compactness, write
	\[
	R^{(k)}_{m,s}=\frac{s}{mk+s}\binom{mk+s}{m},\qquad
	F^{(k)}_m=C^{(k)}_{m+1},
	\]
	\[
	P^{(k)}_{m,s}=\prod_{j=0}^{m-1}(s+jk),\qquad
	Q^{(k)}_{m,s}=\prod_{j=1}^{m}\binom{s+jk-1}{k-1},
	\]
	\[
	U^{(k)}_{m,s}=\frac1{m!}\prod_{j=1}^{m}\binom{s+jk-1}{k}.
	\]
	Empty products are interpreted as $1$.
	
	We shall use the following elementary insertion counts.  In each case the suffix is already standardized and belongs to the same avoiding class; after the new block is inserted, all values are standardized again.
	
	\begin{lemma}\label{lem:insertioncounts}
		Let \(L\) be the length of a standardized suffix.
		\begin{enumerate}[label=(\alph*)]
			\item If a new complete block of length \(k\) is inserted immediately to the left and its last entry must be larger than every entry in the suffix, then there are \(\binom{L+k-1}{k-1}\) choices.
			\item Suppose the suffix begins with its minimum.  If a new complete block of length \(k\) is inserted immediately to the left and the first \(k-1\) entries of the new block must be smaller than every entry in the suffix, then there are \(L\) choices.
		\end{enumerate}
	\end{lemma}
	
	\begin{proof}
		For (a), the last entry of the inserted block is forced to be the maximum of the enlarged suffix.  The other \(k-1\) entries of the inserted block may be chosen arbitrarily from the remaining \(L+k-1\) ranks, and are then written increasingly.
		
		For (b), the first \(k-1\) entries of the inserted block must be the \(k-1\) smallest ranks in the enlarged word.  Since the suffix begins with its minimum, the boundary descent holds precisely when the last entry of the inserted block is not the smallest of the remaining \(L+1\) ranks.  This gives \(L\) choices.
	\end{proof}
	
	\begin{lemma}\label{lem:recordlowcount}
		The number of permutations in $\D_{mk+s,k}$ for which every lower boundary entry
		\[
		x_{2,1},x_{3,1},\ldots,x_{m+1,1}
		\]
		is smaller than all entries to its left is
		\[
		\frac1{m!}\prod_{j=1}^{m}\binom{s+jk-1}{k}.
		\]
	\end{lemma}
	
	\begin{proof}
		Let $R_{m,s,k}$ be the poset obtained from $B_{m,s,k}$ by adding
		\[
		x_{i+1,1}<y\qquad (1\le i\le m)
		\]
		for every position $y$ to the left of $x_{i+1,1}$.  Its linear extensions are exactly the permutations with the stated record-low property.  We count these extensions by the hook-length formula for rooted tree posets.
		
		Orient the Hasse diagram from smaller to larger.  The minimal element is $x_{m+1,1}$.  The first entries of the blocks form the spine
		\[
		x_{m+1,1}<x_{m,1}<\cdots <x_{2,1}<x_{1,1},
		\]
		and, from each $x_{i,1}$, there is an attached chain
		\[
		x_{i,1}<x_{i,2}<\cdots <x_{i,k}\qquad(1\le i\le m),
		\]
		while the root has the final attached chain
		\[
		x_{m+1,1}<x_{m+1,2}<\cdots <x_{m+1,s}.
		\]
		Thus $R_{m,s,k}$ is a rooted tree poset.  The subtree sizes of the spine vertices $x_{i,1}$ for $1\le i\le m$ are $ik$, and the subtree size of the root $x_{m+1,1}$ is $N=mk+s$.  The non-spine vertices in each complete block contribute
		\[
		(k-1)(k-2)\cdots 1=(k-1)!,
		\]
		and the non-root vertices in the final block contribute $(s-1)!$.  Hence the hook product is
		\[
		N\,(s-1)!\,(k-1)!^m\prod_{i=1}^m ik
		=N\,(s-1)!\,(k!)^m m!.
		\]
		The rooted-tree hook-length formula\cite{FrameRobinsonThrall}  therefore gives
		\[
		\LE(R_{m,s,k})=
		\frac{N!}{N(s-1)!(k!)^m m!}
		=\frac{(mk+s-1)!}{(s-1)!(k!)^m m!}
		=\frac1{m!}\prod_{j=1}^m\binom{s+jk-1}{k},
		\]
		as required.
	\end{proof}
	
	\begin{lemma}\label{lem:tightening}
		For every $N$ and $k$, avoidance of the classical pattern $132$ in $\D_{N,k}$ is equivalent to avoidance of the vincular pattern $\underline{13}2$.  Similarly, avoidance of the classical pattern $231$ is equivalent to avoidance of $2\underline{31}$.
	\end{lemma}
	
	\begin{proof}
		The vincular occurrence is, of course, a classical occurrence, so only the converse needs proof.  Suppose first that $\pi\in\D_{N,k}$ contains a classical $132$, say $\pi_a<\pi_c<\pi_b$ with $a<b<c$.  Along the positions from $a$ to $b$, the values start below $\pi_c$ and end above $\pi_c$.  The first adjacent step along this interval at which the values cross from below $\pi_c$ to above $\pi_c$ must be an ascent inside a block, because boundary steps in $\D_{N,k}$ are descents.  Thus there is a position $t$ with $\pi_t<\pi_c<\pi_{t+1}$, and $\pi_t,\pi_{t+1},\pi_c$ is an occurrence of $\underline{13}2$.
		
		For $231$, suppose $\pi_a,\pi_b,\pi_c$ is a classical occurrence, so $a<b<c$ and $\pi_c<\pi_a<\pi_b$.  Along the positions from $b$ to $c$, the values start above $\pi_a$ and end below $\pi_a$.  The first crossing from above to below $\pi_a$ must be a boundary descent, since all steps inside blocks are ascents.  Hence for some $t$ we have $\pi_t>\pi_a>\pi_{t+1}$, and $\pi_a,\pi_t,\pi_{t+1}$ is an occurrence of $2\underline{31}$.
	\end{proof}
	
	\begin{lemma}\label{lem:oneadjacentfuss}
		In the complete-block case $N=nk$, the four one-adjacent classes
		\[
		\underline{13}2,\qquad 2\underline{31},\qquad \underline{31}2,\qquad 2\underline{13}
		\]
		are all counted by $C^{(k)}_n$.
	\end{lemma}
	
	\begin{proof}
		For $\underline{13}2$, Lemma~\ref{lem:tightening} reduces the enumeration to the classical $132$ case.  Thus the count is $R^{(k)}_{n-1,k}=C^{(k)}_n$.  The pattern $2\underline{31}$ is similarly equivalent to classical $231$, and so has the same count.  Applying the complete-block reverse-complement symmetry of Lemma~\ref{lem:rc} gives the two reverse-complement partners $\underline{31}2$ and $2\underline{13}$, with the same value.
	\end{proof}
	
	\begin{lemma}\label{lem:first312completion}
		Let $\pi\in\D_{mk+s,k}$ avoid $\underline{31}2$.  Then, whenever $s\ge2$, the entries $x_{m+1,2},\ldots,x_{m+1,s}$ are the largest $s-1$ entries of $\pi$.  Consequently, for $1\le s<k$, appending a new largest entry to the last block gives a bijection
		\[
		\D_{mk+s,k}(\underline{31}2)\longleftrightarrow \D_{mk+s+1,k}(\underline{31}2).
		\]
	\end{lemma}
	
	\begin{proof}
		Suppose that some earlier entry is larger than $x_{m+1,a}$ for an $a\ge2$.  Moving from that earlier position to the position of $x_{m+1,a}$, the word eventually crosses from values larger than $x_{m+1,a}$ to values smaller than $x_{m+1,a}$, since $x_{m+1,1}<x_{m+1,a}$.  Such a downward crossing cannot occur inside a block, where the entries increase; hence it occurs at a boundary descent, say
		\[
		x_{i,k}>x_{m+1,a}>x_{i+1,1}.
		\]
		Then $x_{i,k},x_{i+1,1},x_{m+1,a}$ is an occurrence of $\underline{31}2$, a contradiction.  Thus no earlier entry is larger than any of $x_{m+1,2},\ldots,x_{m+1,s}$, and these entries are the largest $s-1$ values.
		
		Appending a new largest entry cannot create a new $\underline{31}2$ occurrence, because the new entry is in the last position and is larger than all old entries; it can be neither the adjacent lower element nor the later middle element of such an occurrence.  Deleting the final largest entry gives the inverse map.
	\end{proof}
	
	\begin{lemma}\label{lem:last213completion}
		Let $\pi\in\D_{mk+s,k}$ avoid $2\underline{13}$.  Then the entries of the final block form an interval of consecutive values.  Consequently, for $1\le s<k$, inserting after $x_{m+1,s}$ a new value immediately above $x_{m+1,s}$ in value order gives a bijection
		\[
		\D_{mk+s,k}(2\underline{13})\longleftrightarrow
		\D_{mk+s+1,k}(2\underline{13}).
		\]
	\end{lemma}
	
	\begin{proof}
		If two adjacent entries $x_{m+1,a}<x_{m+1,a+1}$ of the final block had an earlier value $y$ strictly between them, then $y,x_{m+1,a},x_{m+1,a+1}$ would be an occurrence of $2\underline{13}$.  Hence no value outside the final block lies between two adjacent final-block values, and the final block is an interval.
		
		Let $b=x_{m+1,s}$, and insert a new final entry of value $b+1$, increasing all old values larger than $b$ by one.  The old relative order is unchanged.  A new occurrence using the inserted entry must use it as the third entry of the adjacent pair, so the second entry would have to be the old final entry $b$.  But then the first entry would need to have value strictly between $b$ and $b+1$, which is impossible.  Deleting the inserted final entry is the inverse operation.
	\end{proof}
	
	\begin{theorem}\label{thm:firstadjacent}
		For vincular patterns in which the first two letters are adjacent, we have the following classification.
		\[
		\begin{array}{c|c}
			\toprule
			\text{pattern} & \text{number of avoiding permutations in }\D_{mk+s,k}\\
			\midrule
			\underline{12}3 & 0\text{ for }k\ge3,\ N\ge3;\ \text{special }k=2\text{ case below}\\[1mm]
			\underline{13}2 & R^{(k)}_{m,s}\\[1mm]
			\underline{21}3 & Q^{(k)}_{m,s}\\[1mm]
			\underline{23}1 & P^{(k)}_{m,s}\\[1mm]
			\underline{31}2 & F^{(k)}_m\\[1mm]
			\underline{32}1 & \LE(P^{321}_{m,s,k})\\
			\bottomrule
		\end{array}
		\]
		For $k=2$, the exceptional first row is
		\[
		|\D_{2m+1,2}(\underline{12}3)|=2^m m!,\qquad
		|\D_{2m+2,2}(\underline{12}3)|=\prod_{j=1}^{m}(2j+1).
		\]
		The poset $P^{321}_{m,s,k}$ is obtained from the block poset $B_{m,s,k}$ by adding
		\[
		x_{i+1,1}<x_{j,a}
		\]
		for every position $(j,a)$ lying strictly to the right of $x_{i+1,1}$.
	\end{theorem}
	
	\begin{proof}
		For $\underline{12}3$ with $k\ge3$ and $N\ge3$, the first three entries lie in the first block and already contain adjacent entries realizing $12$ followed by a larger entry.  When $k=2$, an occurrence can only start with the adjacent ascent inside a complete block.  Avoidance is therefore equivalent to requiring the second entry of each complete block to be larger than every entry to its right.  Inserting complete blocks from right to left gives, at the $j$th step, $s+2j-1$ choices.  This yields $2^m m!$ when $s=1$ and $\prod_{j=1}^m(2j+1)$ when $s=2$.
		
		By Lemma~\ref{lem:tightening}, avoiding $\underline{13}2$ is equivalent in $\D_{N,k}$ to avoiding the classical pattern $132$.  Theorem~\ref{thm:classical132231} therefore gives $R^{(k)}_{m,s}$.
		
		For $\underline{21}3$, the adjacent descent must be a boundary descent.  Thus an occurrence is a boundary descent $x_{i,k}>x_{i+1,1}$ followed by a later entry larger than $x_{i,k}$.  Avoidance is therefore equivalent to requiring every upper boundary entry $x_{i,k}$ to be larger than all entries to its right.  Insert the complete blocks from right to left.  If the current suffix has length $L=s+(j-1)k$, Lemma~\ref{lem:insertioncounts}(a) gives $\binom{s+jk-1}{k-1}$ choices.  Multiplying over $j=1,\ldots,m$ gives $Q^{(k)}_{m,s}$.
		
		For $\underline{23}1$, an occurrence is an adjacent ascent whose first entry is followed later by a smaller entry.  Hence, for every complete block, its first $k-1$ entries must be smaller than all entries to their right.  The suffix produced after deleting any initial complete block begins with its minimum.  Lemma~\ref{lem:insertioncounts}(b), with $L=s+(j-1)k$, gives the recurrence
		\[
		B_{j,s}=(s+(j-1)k)B_{j-1,s},\qquad B_{0,s}=1,
		\]
		so the count is $P^{(k)}_{m,s}$.
		
		For $\underline{31}2$, Lemma~\ref{lem:first312completion} completes the last block bijectively by appending largest entries.  The complete-block value is $F^{(k)}_m$ by Lemma~\ref{lem:oneadjacentfuss}.
		
		Finally, $\underline{32}1$ can occur only at a boundary descent followed by a later smaller entry.  Thus avoidance is equivalent to requiring every lower boundary entry $x_{i+1,1}$ to be smaller than all entries to its right.  This is exactly the displayed poset condition.
	\end{proof}
	
	\begin{theorem}\label{thm:lastadjacent}
		For vincular patterns in which the last two letters are adjacent, we have the following classification.
		\[
		\begin{array}{c|c}
			\toprule
			\text{pattern} & \text{number of avoiding permutations in }\D_{mk+s,k}\\
			\midrule
			1\underline{23} & 0\text{ for }k\ge3,\ N\ge3;\ \text{otherwise }\LE(P^{123}_{m,s,k})\\[1mm]
			1\underline{32} & U^{(k)}_{m,s}\\[1mm]
			2\underline{13} & F^{(k)}_m\\[1mm]
			2\underline{31} & R^{(k)}_{m,s}\\[1mm]
			3\underline{12} & k^m m!\\[1mm]
			3\underline{21} & \LE(P^{321\prime}_{m,s,k})\\
			\bottomrule
		\end{array}
		\]
		The poset $P^{123}_{m,s,k}$ is obtained from the block poset $B_{m,s,k}$ by adding, for every adjacent ascent pair $x_{i,a}<x_{i,a+1}$, the relations
		\[
		x_{i,a}<y
		\]
		for every earlier position $y$.  The poset $P^{321\prime}_{m,s,k}$ is obtained from $B_{m,s,k}$ by adding, for every boundary descent $x_{i,k}>x_{i+1,1}$, the relations
		\[
		y<x_{i,k}
		\]
		for every earlier position $y$.
	\end{theorem}
	
	\begin{proof}
		The pattern $1\underline{23}$ is forced for $k\ge3$ and $N\ge3$, since the first three entries lie in the first block and form such a pattern.  In the remaining cases, $1\underline{23}$ occurs precisely when an adjacent ascent has a smaller earlier entry, so avoidance is exactly the displayed set of poset relations.
		
		For $1\underline{32}$, the adjacent descent must be a boundary descent.  An occurrence consists of an earlier smaller entry followed by such a descent, so each lower boundary entry $x_{i+1,1}$ must be smaller than all entries to its left.  These are exactly the successive record-low conditions counted by Lemma~\ref{lem:recordlowcount}, giving $U^{(k)}_{m,s}$.
		
		For $2\underline{13}$, Lemma~\ref{lem:last213completion} completes the final block bijectively by inserting immediate successors of the final entry.  The complete-block value is $F^{(k)}_m$ by Lemma~\ref{lem:oneadjacentfuss}.
		
		By Lemma~\ref{lem:tightening}, avoiding $2\underline{31}$ is equivalent in $\D_{N,k}$ to avoiding the classical pattern $231$.  Hence Theorem~\ref{thm:classical132231} gives $R^{(k)}_{m,s}$.
		
		For $3\underline{12}$, an occurrence is an adjacent ascent whose larger entry is preceded by a still larger entry.  Thus avoidance is equivalent to requiring every entry that is not the first entry of a block after the first to be a left-to-right maximum at its position.  Equivalently, after deleting the first entry of each block except the first, the remaining word is increasing.  The resulting poset is the same chain-insertion poset that occurred for $\underline{312}$ in Theorem~\ref{thm:fullyconsecutive}: the deleted entries $x_{2,1},\ldots,x_{m+1,1}$ are inserted below $x_{1,k},x_{2,k},\ldots,x_{m,k}$, respectively.  Hence the count is again $\prod_{i=1}^m ik=k^m m!$.
		
		Finally, $3\underline{21}$ occurs exactly when a boundary descent is preceded by a larger entry.  Avoidance is therefore equivalent to requiring each upper boundary entry $x_{i,k}$ to be larger than all entries to its left, which is the displayed poset condition.
	\end{proof}
	
	\begin{corollary}\label{cor:vincularspecializations}
		In the complete-block case $N=(m+1)k$, the four one-adjacent classes
		\[
		\underline{13}2,\qquad 2\underline{31},
		\qquad \underline{31}2,
		\qquad 2\underline{13}
		\]
		are all counted by $C^{(k)}_{m+1}$.  Moreover,
		\[
		\sum_{m\ge0} R^{(k)}_{m,s}z^m=T(z)^s,
		\]
		where $T(z)=1+zT(z)^k$, and the product classes have the explicit rising-step forms
		\[
		P^{(k)}_{m,s}=s(s+k)\cdots(s+(m-1)k),
		\]
		\[
		Q^{(k)}_{m,s}=\prod_{j=1}^m\binom{s+jk-1}{k-1},
		\qquad
		U^{(k)}_{m,s}=\frac{(mk+s-1)!}{(s-1)!(k!)^m m!}.
		\]
	\end{corollary}
	
	\begin{proof}
		The complete-block assertion is Lemma~\ref{lem:oneadjacentfuss} with $n=m+1$.  The generating function for $R^{(k)}_{m,s}$ is \eqref{eq:Raney}.  The formula for $P^{(k)}_{m,s}$ is its definition, and the displayed factorial form for $U^{(k)}_{m,s}$ is the last expression obtained in Lemma~\ref{lem:recordlowcount}.
	\end{proof}
	
	\begin{remark}
		The linear-extension formulas in Theorems~\ref{thm:fullyconsecutive}, \ref{thm:firstadjacent}, and \ref{thm:lastadjacent} are exact finite formulas, not recurrences.  They are comparatively explicit because the defining posets have only the original block relations plus one family of additional inequalities.  In particular, the exceptional cases are still reduced to standard poset enumeration rather than left as unspecified avoidance classes.
	\end{remark}
	
	\bibliographystyle{cas-model2-names}
	\bibliography{cas-refs}
	
	\printcredits
	
\end{document}